\documentclass[11pt,twoside]{article}
\usepackage{a4}
\usepackage{amssymb,amsmath,amsthm,latexsym}
\usepackage{amsfonts}
\usepackage{amsfonts}
\usepackage{graphicx}
\usepackage{nicematrix}
\usepackage{blkarray}

\newtheorem{theorem}{Theorem}[section]

\newtheorem{corollary}[theorem] {Corollary}
\newtheorem{definition}[theorem]{Definition}
\newtheorem{example}[theorem]{Example}
\newtheorem{lemma} [theorem]{Lemma}

\newtheorem{proposition}[theorem]{Proposition}
\newtheorem{remark}[theorem]{Remark}

\usepackage{caption}
\usepackage{tikz}
\usepackage{kbordermatrix}

\renewcommand{\kbldelim}{(}
\renewcommand{\kbrdelim}{)}
\vspace{5cm}

\begin{document}

\label{'ubf'}  
\setcounter{page}{1}                                 

\markboth {\hspace*{-9mm} \centerline{\footnotesize \sc
	On the circuits of splitting matroids representable over $GF(p)$}
}
{ \centerline                           {\footnotesize \sc  
	Prashant Malavadkar$^1$, Uday Jagadale$^2$ and Sachin Gunjal$^3$                                                 } \hspace*{-9mm}              
}

\begin{center}
{ 
	{\Large \textbf { \sc On the circuits of splitting matroids representable over $GF(p)$
		}
	}
	\\

	\medskip

	{\sc Prashant Malavadkar$^1$, Uday Jagadale$^2$ and Sachin Gunjal$^3$}\\
	{\footnotesize  School of Mathematics and Statistics, MIT-World Peace University,
		Pune 411 038,	India. }\\
	{\footnotesize e-mail: {\it 1. prashant.malavadkar@mitwpu.edu.in, 2. uday.jagdale@mitwpu.edu.in,  \\3. sachin.gunjal@mitwpu.edu.in}}
}
\end{center}

\thispagestyle{empty}

\hrulefill

\begin{abstract}  
{\footnotesize 
\noindent We extend the splitting operation from  binary matroids (Raghunathan et al., 1998) to $p$- matroids, where $p$-matroids refer to matroids representable over $GF(p).$  We also characterize circuits, bases, and independent sets of the resulting matroid. Sufficient conditions to yield Eulerian $p$-matroids from Eulerian and non-Eulerian $p$-matroids by applying the splitting operation are obtained. A class of connected $p$-matroids that gives connected $p$-matroids under the splitting operation is characterized.}
\end{abstract}
\hrulefill



\section{Introduction}
Fleischner\cite{fle} characterized Eulerian graphs using the notion of splitting away a pair of edges from a vertex of degree at least three. Raghunathan et al.\cite{rsw} extended the splitting operation from graphs to binary matroids, and characterization of Eulerian binary matroids was obtained using this operation as follows: 


\begin{theorem}\label{the1}
A binary matroid $M$ on a set $E$ is Eulerian if and only if $M$ can be transformed by repeated applications of the splitting operation into a matroid in which $E$ is a circuit.
\end{theorem}

\noindent A sufficient condition was provided by Shikare \cite{shi} to obtain connected binary matroids from a $4$-connected binary matroids using the splitting operation:

\begin{theorem}
If $M$ is a $4$-connected binary matroid on at least $9$ elements and $a, b$ are distinct elements of $M,$ then $M_{a,b}$ is a connected binary matroid.
\end{theorem} 

\noindent The splitting operation and its various properties were discussed in \cite{g, ml, mal, PMT, ma, ba, wag, wu}. 

\noindent In this paper, simple and coloopless matroids representable over $GF(p)$ are discussed. A matroid $M$ is called a $p$-matroid if $M \cong M[A],$ where $M[A]$ is the vector matroid of matrix $A$ of size $m \times n$ over the field $F = GF(p)$ for some prime $p$. We denote the set of column labels of $M$ (viz. the ground set of $M$) by $E$, the set of circuits of $M$ by $\mathcal {C}(M),$ and the set of independent sets of $M$ by $\mathcal {I}(M)$. For undefined and standard terminologies, refer to Oxley \cite{ox}.

\noindent In the present paper, we define the splitting operation on $p$-matroids and give a characterization of the circuits of the resulting $p$-matroid. We give a sufficient condition for the $p$-matroid to be connected after the splitting operation. A sufficient condition to yield Eulerian $p$-matroid from Eulerian $p$-matroid after the splitting operation is also obtained.

	

\section{Splitting Operation on $p$-Matroids }
	
Let $M$ be a $p$-matroid on ground set $E$, $\{a,b\} \subset E $ and $\alpha$ is a non-zero element of $GF(p)$. We define splitting operation on a $p$-matroid $M$ as follows. 

\begin{definition}\label{def1}
	Let $M\cong M[A]$ be a $p$-matroid on ground set $E,$$\{a,b\} \subset E,$ and $\alpha\neq 0$ in $ GF(p)$. The matrix $A_{a,b}$ is constructed  from $A$ by appending an extra row to $A$ which has coordinates equal to $\alpha$ in the columns corresponding to the elements $a,$$b$ and zero elsewhere. Define the splitting matroid $M_{a,b}$ to be the vector matroid  $M[A_{a,b}]$. The transformation of $M$ to $M_{a,b}$ is called the splitting operation.
\end{definition}

\begin{example}\label{ex1}
Consider the vector matroid $M[A]$ of a matrix $A$ over a field $GF(3)$.
	
\begin{center}
	
	$\mathbf{A} = 
	\begin{pNiceMatrix}%
		[first-col,
		first-row,
		code-for-first-col = \color{black},
		code-for-first-row = \color{black}]
		& 1 & 2 & 3 & 4 & 5 & 6 & 7 & 8   \\
		&  1 & 0 & 1 & 0 & 0 & 0 & 1 & 1  \\
		&  0 & 1 & 1 & 0 & 0 & 0 & 1 & 1  \\
		&  0 & 0 & 0 & 1 & 0 & 1 & 1 & 1  \\
		&  0 & 0 & 0 & 0 & 1 & 1 & 1 & 0  \\
	\end{pNiceMatrix} \qquad
\mathbf{A_{3,5}} = 
\begin{pNiceMatrix}%
	[first-col,
	first-row,
	code-for-first-col = \color{black},
	code-for-first-row = \color{black}]
	& 1 & 2 & 3 & 4 & 5 & 6 & 7 & 8   \\
	&  1 & 0 & 1 & 0 & 0 & 0 & 1 & 1  \\
	&  0 & 1 & 1 & 0 & 0 & 0 & 1 & 1  \\
	&  0 & 0 & 0 & 1 & 0 & 1 & 1 & 1  \\
	&  0 & 0 & 0 & 0 & 1 & 1 & 1 & 0  \\
	&  0 & 0 & 2 & 0 & 2 & 0 & 0 & 0  \\
	
\end{pNiceMatrix}$
\end{center}	

\noindent For $a=3$, $b=5$ and $\alpha=2,$ the matrix $A_{3,5},$ represents the splitting matroid $M_{3,5}.$ 

%
%
\noindent The collection of circuits of $M$ and $M_{3,5}$ is given in the following table.
	
		\begin{tabular}{ | m{7cm}| m{7cm} | }
			\hline
			\textbf{~~~~~~~~~~~~~~Circuits of $M$} & \textbf{~~~~~~~~~~~~~~Circuits of $M_{3,5}$} \\ 
			\hline
			$\{1, 2, 4, 8\},\{1, 2, 6, 7\},\{3, 5, 6, 8\},\{4, 6, 7, 8\}$ & $\{1, 2, 4, 8\},\{1, 2, 6, 7\},\{3, 5, 6, 8\},\{4, 6, 7, 8\}$  \\ 
			
			\hline
			$\{1, 2, 4, 5, 7\},\{3, 4, 5, 7\},\{3, 4, 8\}$ & $\{1, 2, 3, 4, 5, 7\},\{3, 4, 5, 6, 7\},\{3, 4, 5, 7, 8\}$ \\ 
			
			\hline
			$\{1, 2, 3\},\{4, 5, 6\},\{5, 7, 8\},\{3, 6, 7\}$ & $\{1, 2, 3, 4, 5, 6\},\{1, 2, 3, 5, 7, 8\}$ \\ 
			
			\hline
			$\{1, 2, 5, 6, 8\}$ & $-$ \\ 
			
			\hline

		\end{tabular}
	
\end{example}

\noindent It is interesting to observe that unlike the splitting in binary matroids, circuits $\{1, 2, 3, 4, 5, 7\},$ $\{3, 4, 5, 6, 7\},$ $\{3, 4, 5, 7, 8\}$ of $M_{3,5}$ are neither circuits of $M$ nor disjoint union of circuits of $M.$

\begin{remark}
	Let $\{a,b\}$ be a cocircuit of $M.$ Then $M_{a,b}\cong M.$ 
\end{remark}

\begin{remark}
	Observe that rank$(A)\le$ rank$(A_{a,b})\le$ rank$(A)+1,$ which implies that $rank(M)\le rank(M_{a,b})\le rank(M) + 1.$
\end{remark}

\begin{remark}
	When $p=2$, the splitting operation coincides with the splitting operation for binary matroids introduced by Raghunathan et al.\cite{rsw}. 
\end{remark}

\noindent In the following discussion, we describe circuits of $M_{a,b}$ with respect to circuits of $M.$ 
\begin{lemma}
 	If $C\in \mathcal{C}(M)$ and $C\cap\{a,b\}=\phi,$ then C is a circuit of $M_{a,b}$.
\end{lemma} 
	
\begin{proof}
	Observe that $C$ is a dependent set of $M_{a,b}$. If $C$ is not a circuit of $M_{a,b},$ then $C$ contains a circuit $C'$ of $M_{a,b}$. Note that $C'$ is a dependent set of $M$ and $C'\subset C$ which is not possible. Therefore $C$ is a circuit of $M_{a,b}.$
\end{proof}

\noindent Let $C_k=\{u_1, u_2,\ldots, u_l : u_i \in E, i=1,2,\dots,l \}$  be a circuit of $M$ for some positive integers $k,$ $l$ and $|C_k \cap \{a,b\}|=2.$ Suppose, without loss of generality, $u_1=a$,$u_2=b$. If there are non-zero constants $a_1^k, a_2^k, ..., a_l^k$  in $GF(p)$ such that $\sum_{i=1}^{i=l} a_i^k u_i \equiv 0 (mod~p)$ and $a_1^k + a_2^k \equiv 0 (mod~p)$ then we call $C_{k}$ a $p$-circuit of $M.$ However, if $|C_k \cap \{a,b\}|=1$ or $|C_k \cap \{a,b\}|=2$ and $a_1^k + a_2^k\not\equiv 0 (mod \: p),$ then we call $C_{k}$ an $np$-circuit of $M$.
 
\noindent We denote $\mathcal {C}_0=\{C\in \mathcal {C}(M): C$ is a $p$-circuit or $C\cap \{a,b\}=\phi\}$.

\begin{remark}

\noindent Let $C$ and $\{a,b\}$ be a circuit and a cocircuit of a $p$-matroid $M,$ respectively. If $p=2,$ then every circuit of $M$ is a circuit of $M_{a,b}.$ But for $p>2$ every circuit of $M$ may not be a circuit of $M_{a,b}.$ Note that $|C\cap \{a,b\}| \neq 1,$ therefore $|C\cap \{a,b\}|$ is even. If $a,b \in C$ and $C$ is an $np$-circuit of $M,$ then $C\in \mathcal{I}(M_{a,b}).$

\end{remark}
   
\begin{lemma}\label{dep}
Let $M$ be a $p$-matroid on ground set $E$ and $\{a,b\}\subset E$. If $C_1$ and $C_2$ are disjoint $np$-circuits of $M$ such that $a \in C_1, b \in C_2,$ then $C_1 \cup C_2 $ is a dependent set of $M_{a,b}$.
\end{lemma}

\begin{proof}
Let $C_1 = \{u_1,u_2,\ldots,u_l\}$ and $C_2= \{v_1, v_2,\ldots,v_k\}.$ Without loss of generality, assume $u_1=a$ and $v_1=b$.
Since $C_1$ and $C_2$ are circuits, there exist non-zero constants $\alpha_1,\alpha_2,\ldots,\alpha_l$ and $\beta_1,\beta_2,\ldots,\beta_k$ such that 

\begin{equation}\label{eq1}
 \alpha_1 u_1+ \alpha_2 u_2+ \ldots+\alpha_l u_l \equiv 0(mod\:p)
\end{equation}
\begin{equation}\label{eq2}
\beta_1 v_1+\beta_2 v_2+ \dots+\beta_k v_k \equiv 0(mod \:p).
\end{equation}
	
\noindent Multiplying equation (\ref{eq1}) by $\alpha_1^{-1}$ and equation (\ref{eq2}) by $-\beta_1^{-1}$ we get
\[u_1+ \alpha_2' u_2+...+\alpha_l' u_l \equiv 0 (mod \: p),\]
\[-v_1+\beta_2' v_2+...+\beta_k' v_k \equiv 0 (mod \: p).\]
	
\noindent Thus we get non-zero scalars $1, \alpha_i', i= 2,3,\ldots,l$ and $-1, \beta_j', j= 2,3,\ldots,k$ such that
	
\[u_1+ \alpha_2' u_2+...+\alpha_l' u_l+ (-1)v_1+\beta_2' v_2+...+\beta_k' v_k \equiv 0\] and 
\[ coeff. (u_1)+coeff.(v_1) \equiv 0 (mod \: p).\]
	
\noindent Therefore $C_1 \cup C_2$ is a dependent set of $M_{a,b}.$
\end{proof}

\begin{remark}\label{L1}	
\noindent Let $C_1$, $C_2$ be disjoint $np$-circuits of $M$ such that $a \in C_1, b \in C_2$ and $I \neq \phi$ be an independent set of $M.$ Then $C_1 \cup C_2 \cup I $ can not be a circuit of $M_{a,b}$ as it contains the dependent set $C_1 \cup C_2$ of $M_{a,b}$.	
\end{remark}

\noindent Consider subsets of $E$ of the type $C\cup I$ where $C=\{u_1,u_2,\ldots,u_l\}$ is an $np$-circuit of $M$ which is disjoint from an independent set $I=\{v_1,v_2,\ldots,v_k\}$ and $\{a,b\}\subset (C\cup I)$. We say $C\cup I$ is $p$-dependent if it contains no member of $\mathcal {C}_0$ and there are non-zero constants $\alpha_1,\alpha _2,\ldots,\alpha_l$ and $\beta_1, \beta_2,\ldots,\beta_k$ such that $\sum_{i=0}^{l}\alpha_i u_i + \sum_{j=0}^{k}\beta_j v_j = 0 (mod \: p)$ and $coeff.(a)+coeff.(b) = 0(mod \: p).$

\noindent Theorem \ref{circuits} characterizes the circuits of the splitting matroid $M_{a,b}$ with respect to circuits of $M.$

\begin{theorem}\label{circuits}
Let $M$ be a $p$-matroid on ground set $E$ and $\{a,b\}\subset E(M)$. Then $\mathcal {C}(M_{a,b})=\mathcal{C}_0\cup \mathcal {C}_1\cup\mathcal {C}_2$ where
	
\begin{description}
	
\item $\mathcal {C}_0=\{C\in \mathcal {C}(M): C$ is a $p$-circuit or $C\cap \{a,b\}=\phi\}$;
		
\item $\mathcal{C}_1=$ Set of minimal members of $\{C \cup I : (C \cup I)$ is $p$-dependent and there are no disjoint $np$-circuits $C_1', C_2'$ such that $C_1' \cup C_2' \subset ( C\cup I)$\}; and
		
\item 	$\mathcal {C}_2=$Set of minimal members of $ \{C_1\cup C_2 : C_1, C_2 \in \mathcal {C}(M) , a\in C_1,$ $b\in C_2,$ $C_1\cap C_2=\phi$ and there is no $C\in \mathcal {C}_0\cup \mathcal {C}_1$ such that $C\subset C_1 \cup C_2$\}.
	
\end{description}	
\end{theorem}
\begin{proof}
By the Definition \ref{def1} and Lemma \ref{dep}, it is clear that every element of $\mathcal{C}_0\cup \mathcal {C}_1\cup\mathcal {C}_2$ is a circuit of $M_{a,b}.$ Conversely, let $C \in \mathcal{C}(M_{a,b}),$ then $C$ is a dependent set of $M$.
	
\noindent $\mathbf{Case~I:}$ If $C$ is a minimally dependent set of $M,$ then either $C\cap \{a,b\}=\phi$ or $C$ is a $p$-circuit of $M.$ That is, $C\in \mathcal {C}_0.$ 
	
\noindent $\mathbf{Case~II:}$ If $C$ is not a minimally dependent set of $M,$ then it contains a minimally dependent set, say $C_1.$ Note that $C_1\notin \mathcal{C}_0.$ Thus $C$ contains no member of $\mathcal {C}_0.$
		
\noindent If $a,b \in C_1,$ then $C\setminus C_1$ is an independent set of $M;$ otherwise $C\setminus C_1$ contains a circuit, say $C',$ such that $C' \in \mathcal {C}_0$ which is not possible. Thus $C=C_1\cup I$ where $C_1$ is an $np$-circuit, $I = C \setminus C_1 \in \mathcal{I}(M)$ and it contains no member of $\mathcal {C}_0.$ Therefore $C$ is a minimal $p$-dependent set of $M$ and it does not contain a disjoint union of two $np$-circuits. Thus $C \in \mathcal{C}_1.$
	
\noindent If $a\in C_1,$ $b\in C\setminus C_1$ and $C\setminus C_1$ is an independent set of $M,$ denoted by $I$, then $C=C_1\cup I$. Note that $C_1\cup I$ is a minimal $p$-dependent set of $M$ and it does not contain a disjoint union of two $np$-circuits. Thus $C \in \mathcal{C}_1.$
	
\noindent If $a\in C_1,$ $b\in C\setminus C_1$ and $C\setminus C_1$ is dependent set of $M$, then there is a circuit, say $C_2,$ contained in $C\setminus C_1$ and $b\in C_2.$ Therefore $C=C_1\cup C_2 \cup I$ where both $C_1$ and $C_2$ are $np$-circuits and $I= C \setminus (C_1 \cup C_2)$ is an independent set of $M.$ By Remark \ref{L1}, $I=\phi$. Further, $C=C_1\cup C_2$ does not contain $p$-dependent set of $M$ and a disjoint union of two $np$-circuits. Therefore $C \in \mathcal {C}_2.$
		
\end{proof}

\noindent In the Example \ref{ex1}, $\mathcal{C}(M_{3,5})=\mathcal{C}_0\cup \mathcal {C}_1\cup\mathcal {C}_2$ where 

$\mathcal{C}_0=\{\{1, 2, 4, 8\},\{1, 2, 6, 7\},\{3, 5, 6, 8\},\{4, 6, 7, 8\}\};$

$\mathcal {C}_1=\{\{1, 2, 3, 4, 5, 7\},\{3, 4, 5, 6, 7\},\{3, 4, 5, 7, 8\}\};$

$\mathcal {C}_2=\{\{1, 2, 3, 4, 5, 6\},\{1, 2, 3, 5, 7, 8\}\}.$

\subsection{Independent sets, Bases and Rank function of $M_{a,b}$}
Now we describe independent sets, bases and the rank function of $M_{a,b}$ in terms of independent sets, bases and the rank function of $M,$ respectively.

\noindent Let $\mathcal{I}_0=\mathcal{I}(M)$ and $\mathcal{I}_1=\{C\cup I : (C\cup I)$ is not $p$-dependent and it contains no union of two disjoint $np$-circuits\}.
		

\begin{proposition}
$\mathcal{I}(M_{a,b})=\mathcal{I}_0\cup \mathcal{I}_1$.
\end{proposition}
\begin{proof}
Clearly $\mathcal{I}_0\cup \mathcal{I}_1 \subseteq \mathcal{I}(M_{a,b}). $ Conversely, assume $ S \in \mathcal{I}(M_{a,b}).$ If $S$ is an independent set of $M,$ then $S \in \mathcal{I}_0 $. If $S$ is dependent set of $M,$ then it contains an $np$-circuit of $M.$ In the light of Lemma \ref{dep}, $S$ does not contain union of two disjoint $np$-circuits. Therefore $S = C \cup I$  for some $C \in \mathcal{C}(M),$ $I \in \mathcal{I}(M)$ and $C \cup I$  is not $p$-dependent. Hence $S \in \mathcal{I}_1. $
	
\end{proof}

\noindent We use $r$ and $r'$ to denote the rank functions of matroid $M$ and $M_{a,b}$ respectively.
\begin{theorem}\label{k}
Let $M$ be a $p$-matroid and $\{a,b\} \subset E.$ If $M$ contains an $np$-circuit, then
$\mathcal{B}(M_{a,b})=\{B\cup e:B\in \mathcal{B}(M), e\notin B$ and $B\cup e$ contains neither $p$-circuit nor $p$-dependent set\}.	
\end{theorem}

\begin{proof}
Let $B\in \mathcal{B}(M)$ and $e\notin B$. If $B\cup e$ contains no $p$-circuit and no $p$-dependent set, then $B\cup e$ is an independent set of $M_{a,b}.$ Moreover, $B\cup e$ is a maximal independent set of $M_{a,b}$ because $r'(M_{a,b}) \leq r(M)+1.$

\noindent Conversely, let $B'$ be a basis of $M_{a,b}.$ Note that $B'$ is an independent subset of $M$ if and only if $M$ contains no $np$-circuit. Since $M$ contains an $np$-circuit, $B'$ is a dependent set of $M.$ If $B'$ contains a $p$-circuit, then $B'$ becomes dependent set of $M_{a,b},$ a contradiction. Therefore $B'$ contains an $np$-circuit, say $C.$ Thus $B'= C\cup I$ where $I = B'\setminus C.$ Note that $I$ is an independent set of $M,$ otherwise there is a $np$-circuit, say $C_1,$ of $M$ contained in $I$ and $C\cup C_1$ gives a dependent subset of $B'$ in $M_{a,b},$ a contradiction. 

\noindent Let $e\in C$ and  $B=B'\setminus e=(C\setminus e)\cup I.$ Now if $B$ is a dependent set of $M,$ then $r[(C\setminus e)\cup I)]<r(M).$ It implies $r[(C\cup I)]<r(M)$ and $r'(C\cup I) \leq r(M).$ That is, $r'(B')\leq r(M),$ a contradiction. Therefore $B=B'\setminus e $ is an independent set of $M.$ Moreover, $|B|=r(M),$ which implies that $B$ is a basis of $M.$ Thus $B'=B\cup e.$ 
\end{proof}

\noindent In the following corollary, we provide the rank function of $M_{a,b}$ with respect to the rank function of $M.$
\begin{corollary}\label{s}
	 Suppose $S\subseteq E(M_{a,b}).$ Then
	\begin{equation}
		\begin{split}
			r'(S) &= r(S) , \text { ~~~~~  if S contains no np-circuit of M; and}\\
			&= r(S)+1,\text {~   if S contains an np-circuit of M.}
		\end{split}
	\end{equation}
\end{corollary}

\begin{proof}
Let $B$ and  $B'$ be the bases of $M|_S$ and $M_{a,b}|_S,$ respectively, then $r(S)=|B|$ and $r'(S)=|B'|.$ If $S$ contains no $np$-circuit, then $|B'|=|B|.$ Therefore $r(S)=r'(S).$ If $S$ contains an $np$-circuit, then by the Theorem \ref{k}, $B'=B \cup  e$ for some $B \in \mathcal{B}(M|_S),$$e\in (S\setminus B)$ and $B\cup e$ contains a unique $np$-circuit which is the fundamental circuit of $e$ with respect to $B.$ Therefore $r'(S)=r'(M_{a,b}|_S)=|B'|=|B\cup e|=|B|+1=r(S)+1.$
\end{proof}

\begin{corollary}\label{coc}
	Let $M$ be a $p$-matroid and $\{a,b\}\subset E$, then
	\begin{enumerate}
		\item $r'(M_{a,b}) = r(M)$ if and only if every circuit $C\in \mathcal{C}(M)$ is a $p$-circuit. Moreover, $M \cong M_{a,b}$.
		\item If M contains an $np$-circuit, then $r'(M_{a,b}) = r(M) + 1.$ 
	\end{enumerate}
	
\end{corollary}
 
\begin{remark}
Let $M$ be a $p$-matroid and $\{a,b\}\subset E.$ If $M$ contains an $np$-circuit, then $\{a,b\}$ is a cocircuit of $M_{a,b}.$ Assume the contrary, that is, $\{a,b\}$ is not a cocircuit of $M_{a,b}.$ Then there is a basis $B'$ of $M_{a,b}$ such that $B'\cap \{a,b\}=\phi.$ By Theorem \ref{k}, $B'= B \cup e$  for some $B \in \mathcal{B}(M)$ and $e \in E\setminus B.$ Now $B'=B \cup e$ contains a unique $np$-circuit of $M$ which is a contradiction to the assumption that there is a basis $B'$ of $M_{a,b}$ such that $B'\cap \{a,b\}=\phi.$ Therefore, $\{a,b\}$ is a cocircuit of $M_{a,b}.$
\end{remark}

\section{Connectivity of Splitting Matroids}

Let $M$ be a matroid having ground set $E.$ The $1$-separation of the matroid $M$ is a partition $(S, T)$ of $E$ such that, $|S|, |T|\geq 1$ and $r(S) + r(T)- r(M) <1.$ We say $M$ is connected if $M$ has no $1$-separation. 

\noindent Observe that in Example \ref{ex1}, the matroid $M$ as well as the splitting matroid $M_{3,5}$ are connected. In general, the connectivity of a $p$-matroid is not closed under the splitting operation, which is observed in the following example.

\begin{example}\label{ex2}
	Consider the vector matroid $M \cong M[A]$ represented by the matrix $A$ over the field $GF(3)$.
	
	\begin{center}
		$\mathbf{A} = 
		\begin{pNiceMatrix}%
		[first-col,
		first-row,
		code-for-first-col = \color{black},
		code-for-first-row = \color{black}]
		& 1 & 2 & 3 & 4 & 5 & 6 & 7 & 8 \\
		& 1 & 0 & 0 & 0 & 0 & 1 & 1 & 2 \\
		& 0 & 1 & 0 & 0 & 1 & 0 & 1 & 1 \\
		& 0 & 0 & 1 & 0 & 1 & 1 & 0 & 1 \\
		& 0 & 0 & 0 & 1 & 2 & 1 & 1 & 0 \\
		& 0 & 1 & 1 & 0 & 0 & 0 & 0 & 0
		\end{pNiceMatrix} \qquad
		\mathbf{A_{1,4}} = 
		\begin{pNiceMatrix}%
		[first-col,
		first-row,
		code-for-first-col = \color{black},
		code-for-first-row = \color{black}]
		& 1 & 2 & 3 & 4 & 5 & 6 & 7 & 8 \\
		& 1 & 0 & 0 & 0 & 0 & 1 & 1 & 2 \\
		& 0 & 1 & 0 & 0 & 1 & 0 & 1 & 1 \\
		& 0 & 0 & 1 & 0 & 1 & 1 & 0 & 1 \\
		& 0 & 0 & 0 & 1 & 2 & 1 & 1 & 0 \\
		& 0 & 1 & 1 & 0 & 0 & 0 & 0 & 0 \\
		& 1 & 0 & 0 & 1 & 0 & 0 & 0 & 0	
	\end{pNiceMatrix}$
	\end{center}
	
	\noindent Here the matroid $M$ is connected. But the splitting matroid $M_{1,4} \cong M [A_{1,4}] $ is not connected. 
\end{example}

\begin{remark}
Let $M$ be a $p$-matroid on ground set $E$ and $\{a,b\}\subset E.$ If $M$ contains no $np$-circuit, then $M_{a,b}\cong M.$ In this case, $M$ is connected if and only if $M_{a,b}$ is connected.
\end{remark}

\noindent In the next result, a sufficient condition is provided for the matroid $M_{a,b}$ to be connected.
\begin{lemma}\label{con}
Let $M$ be a connected $p$-matroid and $\{a,b\}\subset E(M)$. If for every proper subset $S$ of $E(M)$ with $|S|\geq 1$, either $S$ or $T = E(M)\setminus S$ contains an $np$-circuit of $M,$ then $M_{a,b}$ is connected.
\end{lemma}

\begin{proof}
Note that, $M$ is a connected $p$-matroid,which implies that $M$ has no $1$-separation. On the contrary, assume $M_{a,b}$ is not connected. That is, $M_{a,b}$ has $1$-separation,say $(S,T).$ Therefore \[r'(S) + r'(T) -r'(M_{a,b}) < 1.\] If $S$ and $T$ both contains $np$-circuits then, by lemma \ref{s},  we have\[r(S)+1+r(T)+1-r(M)-1 = r(S)+r(T)-r(M) + 1 < 1.\] \noindent Thus we get a contradiction to the fact $r(S)+r(T)-r(M)\geq 0. $ Further, if only one of $S$ or $T,$ say $S,$ contains an $np$- circuit, then we have\[r'(S) + r'(T) -r'(M_{a,b}) = r(S)+1+r(T)-r(M)-1 < 1.\] That is  \[r(S)+r(T)-r(M) < 1.\] Thus $(S,T)$ gives a $1$-separation of $M$ which is not possible.
\end{proof}

\noindent For $n\geq 2,$ a matroid $M$ is said to be vertically $n$-connected if for any positive integer $k<n,$ there does not exist a partition $(S,T)$ of $E(M),$ such that $r(S)+r(T)-r(M)< k,$ and $r(S), r(T)\geq k.$ 

\begin{theorem}
Let $M$ be a connected, vertically $3$-connected, simple $p$-matroid and $\{a,b\}\subset E.$ Then $M_{a,b}$ is connected $p$-matroid if and only if for every $e\in E$ there is an $np$-circuit of $M$ not containing $e.$ 
\end{theorem}

\begin{proof}
On the contrary, suppose that $M_{a,b}$ is not connected. And $(S,T)$ is a $1$-separation of $E(M_{a,b})$. Consequently, $|S|, |T|\geq 1$ and $r'(S)+r'(T)-r'(M_{a,b})\leq 0.$ Following are the possible cases:

\noindent $\mathbf{Case~I:}$ $|S|=1$

\noindent Suppose $S=\{x\}$ where $x\in E$. As  $M$ contains an $np$-circuit $C$ such that $\{x\}\cap C=\phi,$ which gives $C\subseteq T.$ By Corollary \ref{s}, $r'(T)=r(T)+1.$ We get $r(S)+r(T)+1-r(M)-1\leq 0.$ Consequently, $r(S)+r(T)-r(M)\leq 0$ and $|S|,|T|\geq 1,$ which forms a $1$-separation of $M$, a contradiction.

\noindent $\mathbf{Case~II:}$ $|S|,|T|>1$

\noindent If either $S$ or $T$ contains an $np$-circuit then using Lemma \ref{con} we conclude that $M_{a,b}$ is connected.

\noindent Suppose both $S$ and $T$ do not contain an $np$-circuit. Then $r'(S)+r'(T)-r'(M_{a,b})\leq 0.$ Using Corollary \ref{s} we get, $r(S)+r(T)-r(M)-1\leq 0.$ That is $r(S)+r(T)-r(M)\leq 1.$ Since $M$ is simple, $r(S)$, $r(T) \geq 2$. It implies $(S, T)$ is a vertical $2$-separation of $M,$ a contradiction.

\noindent If each of $S$ and $T$ contains an $np$-circuit, then $r'(S)+r'(T)-r'(M_{a,b})\leq 0.$ Again by Corollary \ref{s} we get $r(S)+1+r(T)+1-r(M)-1\leq 0.$ That is $r(S)+r(T)-r(M)\leq -1.$ Consequently,  $(S, T)$ is a $1$ separation of $M$, which is a contradiction. 

\noindent Therefore $M_{a,b}$ has no $1$-separation. We conclude that $M_{a,b}$ is a connected $p$-matroid.

\noindent To check the necessity of the condition, suppose that $M_{a,b}$ is connected. On the contrary, assume that there is an element $e \in E $ which is contained in every $np$-circuit of $M.$ Let $S=\{e\}$ and $T=E\setminus S,$ then $T$ contains no $np$-circuit of $M.$ Thus $r'(S)=1,$ $r'(T)=r(T)=r(M)$ and $|S|, |T|\geq 1.$ Further, $r'(S)+r'(T)-r'(M_{a,b})=1+r(T)-r(M)-1=0.$ Thus $(S,T)$ forms a $1$-separation of $M_{a,b},$ a contradiction. Therefore we conclude that for every $e\in E(M)$ there is an $np$-circuit of $M$ not containing $e.$
\end{proof}

\noindent  A matroid $M$ is connected if and only if, for every pair of distinct elements of $E$, there is a circuit containing both. This result is used to prove the following proposition.

\begin{proposition}\label{npuni}
Let $M$ be a non-connected $p$-matroid on ground set $E$ with exactly two connected components $M_1,$$M_2,$ and $\{a,b\}\subset E$. Let $a \in M_1, b \in M_2$. If $C$ and $C'$ are $np$-circuits of $M$ contained in $M_1$ and $M_2,$ respectively, then $C \cup C'$ is a circuit of $M_{a,b}.$
\end{proposition}

\begin{proof}
By Lemma \ref{dep}, $C \cup C'$ is a dependent set of $M_{a,b}.$ Therefore there is a circuit $C_1$ of $M_{a,b}$ such that $C_1 \subseteq C \cup C'.$ Note that $C_1$ is a dependent set of $M$. 

\noindent Suppose $ C_1 \subset C \cup C'.$ Then $C_1 \nsubseteq C$ and $C_1 \nsubseteq C'$ because $C$ and $C'$ are minimal dependent sets of $M$. Let $C_1 \cap C \neq \phi$ and $C_1 \cap C' \neq \phi$. Then we have the following two cases.

\noindent $\mathbf{Case~I:}$ Let $C_1 \cap \{a,b\} = \phi.$ Then $C_1$ is a $p$-circuit containing elements of $M_1$ and $M_2$ which imply that $M$ is a connected matroid, a contradiction. 

\noindent $\mathbf{Case~II:}$ Let $C_1 \cap \{a,b\} \neq \phi.$ Then $a,b \in C_1$ because $\{a,b\}$ is a co-circuit of $M_{a,b}.$ Since $M$ is not connected, $C_1$ can not be a circuit of $M.$ Thus there is a circuit, say $C_2$, of $M$ such that $C_2 \subset C_1.$ Now by similar argument as above, $C_2 \nsubseteq C$ and $C_2 \nsubseteq C'$.  Consequently, $C_2 \cap C \neq \phi$ and  $C_2 \cap C' \neq \phi$. Thus $C_2$ is a circuit of $M$ containing elements of $M_1$ and $M_2.$ Therefore $M$ is connected which is a contradiction. 

\noindent Hence, in either case, we get a contradiction. Thus $C_1 = C \cup C'.$ 

\end{proof}

\begin{corollary}
Let $M$ be a non-connected $p$-matroid on ground set $E$ with exactly two connected components $M_1,$$M_2$ and $\{a,b\}\subset E$. Let $a \in M_1,$$ b \in M_2$. Then $M_{a,b}$ is a connected $p$-matroid. 
\end{corollary}
\begin{proof}
It is enough to show that for every pair $x,y \in E,$ there is a circuit of $M_{a,b}$ containing $x$ and $y$.
	
\noindent $\mathbf{Case~I:}$ Let $x,y \in M_1.$ Then there is a circuit, say $C,$ of $M$ containing $x,y.$ If $C$ is a $p$-circuit, then it is the desired circuit of $M_{a,b}.$ If $C$ is an $np$-circuit of $M$, then $a \in C.$ Let $C'$ be an $np$-circuit of $M$ contained in $M_2.$ Then by Proposition \ref{npuni}, $C \cup C'$ is a circuit of $M_{a,b}$ containing $x$ and $y.$	

\noindent $\mathbf{Case~II:}$ Let $x \in M_1$ and $y\in M_2$. As $M_1,$ $M_2$ are connected there are circuits, say $C_1,$ $C_2$, containing $x,a$ and $y,b$, respectively, in $M$. Note that $C_1$, $C_2$ are $np$-circuits. By Proposition \ref{npuni}, $C_1\cup C_2$ is a circuit of $M_{a,b}$ containing $x$ and $y.$
	
\end{proof}

\section{Applications}

Eulerian $p$-matroids are not closed under the splitting operation as observed in the following example.
\begin{example}
	 $M=M[A]$ is Eulerian $p$-matroid over $GF(3)$ but the splitting matroid $M_{1,2}=M[A_{1,2}]$ is not Eulerian $p$-matroid.
	\begin{center}
		$\mathbf{A} = 
		\begin{pNiceMatrix}%
			[first-col,
			first-row,
			code-for-first-col = \color{black},
			code-for-first-row = \color{black}]
			& 1 & 2 & 3 & 4 & 5 & 6 & 7 & 8 \\
			& 1 & 0 & 0 & 0 & 0 & 1 & 1 & 2 \\
			& 0 & 1 & 0 & 0 & 1 & 0 & 1 & 1 \\
			& 0 & 0 & 1 & 0 & 1 & 1 & 0 & 1 \\
			& 0 & 0 & 0 & 1 & 2 & 1 & 1 & 0
			
		\end{pNiceMatrix} \qquad
		\mathbf{A_{1,2}} = 
		\begin{pNiceMatrix}%
			[first-col,
			first-row,
			code-for-first-col = \color{black},
			code-for-first-row = \color{black}]
		& 1 & 2 & 3 & 4 & 5 & 6 & 7 & 8 \\
		& 1 & 0 & 0 & 0 & 0 & 1 & 1 & 2 \\
		& 0 & 1 & 0 & 0 & 1 & 0 & 1 & 1 \\
		& 0 & 0 & 1 & 0 & 1 & 1 & 0 & 1 \\
		& 0 & 0 & 0 & 1 & 2 & 1 & 1 & 0	\\
		& 1 & 1 & 0 & 0 & 0 & 0 & 0 & 0
		
		\end{pNiceMatrix}$
	\end{center}

\end{example}

%

\noindent Let $M$ be Eulerian matroid on ground set $E.$ Then there are disjoint circuits $C_1,$$C_2,$ $\ldots$,$C_k$ of $M$ such that
$E= C_1 \cup C_2 \cup ...\cup C_k.$
\noindent Let $\{a,b\} \subset E$ and $M_{a,b}$ be the splitting matroid. We say that the collection \~{C} $=\{C_1,C_2,...,C_k\}$ is a \textit{$p$-decomposition}  of $M$ if $a,b \in C_i,$ for some $i \in \{1,2,\ldots,k\},$ and $C_i$ is a $p$-circuit or $a \in C_i, b \in C_j$ for some $i \neq j$ in $\{1,2,\ldots,k\}$ and $(C_i \cup C_j) \in \mathcal{C}_2$. 


\noindent In the following result we give a sufficient condition to yield Eulerian $p$-matroids from Eulerian $p$-matroids after the splitting operation.

\begin{proposition}\label{e1}
Let $M$ be a $p$-matroid on ground set $E$ and $\{a,b\} \subset E$. If $M$ is Eulerian matroid having a $p$-decomposition, then $M_{a,b}$ is Eulerian. 
\end{proposition}

\begin{proof}
Let \~{C} $=\{C_1,C_2,\ldots,C_k\}$ be a $p$-decomposition of $M.$ Then \[E= C_1 \cup C_2 \cup \ldots\cup C_k. \]

\noindent $\mathbf{Case~I:}$ If  $a,b \in C_i$ some $i \in \{1,2,\ldots,k\} $ and $C_i$ is a $p$-circuit. Then the same collection \~{C} is a disjoint circuit decomposition of $E(M_{a,b})$. Therefore $M_{a,b}$ is Eulerian matroid.

\noindent $\mathbf{Case~II:}$ Assume, without loss of generality, $a \in C_1,$$ b \in C_2$. By hypothesis, \~{C} is a \textit{$p$-decomposition}, therefore $(C_1 \cup C_2)$ is a circuit of $M_{a,b.}$  Denote $(C_1 \cup C_2)$ by $C.$ Then the collection $\{C,C_3,\ldots,C_k\}$ is a disjoint circuit decomposition of $M_{a,b}.$ Hence $M_{a,b}$ is Eulerian matroid.

\end{proof}

\begin{proposition}\label{e2}
If $M_{a,b}$ is Eulerian $p$-matroid having a disjoint circuit decomposition which contains no member of $\mathcal{C}_1,$ then $M$ is Eulerian.
\end{proposition}

\begin{proof}
Let $E(M_{a,b})=  C_1 \cup C_2 \cup \ldots \cup C_k$ be a disjoint circuit decomposition which contains no member of $\mathcal{C}_1.$ If $C_i \in \mathcal{C}_0,$  $C_j \in \mathcal{C}_2$ for some $i,j \in \{1,2, \ldots, k\},$ then $C_i \in \mathcal{C}(M)$ and $C_j = C_j^1 \cup C_j^2$ where $C_j^1, C_j^2 \in \mathcal{C}(M)$ and $C_j^1 \cap C_j^2 = \phi. $ Therefore \[E(M)=  C_1 \cup C_2 \cup ...\cup C_i \cup \ldots\cup C_j^1 \cup C_j^2 \cup \ldots\cup C_l\] is a disjoint circuit decomposition of $E(M).$ Therefore $M$ is Eulerian.
\end{proof}

\noindent Applying Proposition \ref{e1} and \ref{e2} for $p=2,$ we obtain the following result of Raghunthan et al. \cite{rsw} for binary matroids.

\begin{corollary}
Let $M$ be a binary matroid and $a,b \in E.$ Then $M$ is Eulerian if and only if $M_{a,b}$ is Eulerian.
\end{corollary} 

\noindent Note that, the splitting operation on non-Eulerian $p$-matroids may yield Eulerian $p$-matroids.

\begin{example} \label{counter}
Consider the vector matroid $M \cong M[A]$ represented by the matrix $A$ over field $GF(5)$.
	
\begin{center}
	$\mathbf{A} = 
\begin{pNiceMatrix}%
	[first-col,
	first-row,
	code-for-first-col = \color{black},
	code-for-first-row = \color{black}]
	& 1 & 2 & 3 & 4 & 5 \\
	&  1 & 0 & 0 & 1 & 1\\
	&  0 & 1 & 0 & 1 & 1 \\
	&  0 & 0 & 1 & 1 & 0 \\
	 
\end{pNiceMatrix} \qquad
\mathbf{A_{3,5}} = 
\begin{pNiceMatrix}%
	[first-col,
	first-row,
	code-for-first-col = \color{black},
	code-for-first-row = \color{black}]
	& 1 & 2 & 3 & 4 & 5 \\
	&  1 & 0 & 0 & 1 & 1\\
	&  0 & 1 & 0 & 1 & 1 \\
	&  0 & 0 & 1 & 1 & 0 \\
	&  0 & 0 & 1 & 0 & 1 \\
	
\end{pNiceMatrix}$
\end{center}

\noindent The circuits of $M$ are $\{1,2,3,4\},\{1,2,5\},\{3,4,5\}.$ Note that $M$ is non-Eulerian matroid. For $a=3$, $b=5$ and $\alpha=1,$ $A_{3,5}$ represents the splitting matroid $M_{3,5}.$ Here the matroid $M_{3,5}$ has only one circuit $\{1,2,3,4,5\}$. Therefore it is Eulerian matroid over $GF(5).$
\end{example}

%
\noindent In the following result, we provide a sufficient condition for the splitting operation on non-Eulerian $p$-matroids to yield Eulerian $p$-matroids.

\begin{proposition}
Let $M$ be non-Eulerian $p$-matroid on ground set $E.$ Then $E$ can be written as follows: \[E= C_1 \cup C_2 \cup \ldots C_k \cup I\] where $C_i \in \mathcal{C}(M),$$ I \in \mathcal{I}(M), $$C_i \cap C_j = \phi$ when $i \neq j \in \{1,2, \ldots, k\} $ and $C_i \cap I = \phi$ for all $i \in \{1,2, \ldots, k\}.$ Let $\{a,b\} \subset E.$ If there is a circuit $C_i$ such that $C_i \cup I \in \mathcal{C}_1,$ then $M_{a,b}$ is Eulerian Matroid.
\end{proposition}

\noindent An Eulerian $p$-matroid $M$ on $E$ can be transformed by repeated applications of splitting operations into a matroid in which $E$ is a circuit, which is observed in the following example.
\begin{example}
The matrix $A$ gives $GF(3)$ representation of the Eulerian matroid $S(5,6,12).$ Let $M$ denote the vector matroid of $A.$ Consider the following sequence of the splitting operations: $M^1=M_{1,6}$, $M^{2}=M^1_{1,4},$ $M^{3}= M^{2}_{1,3},$ $M^{4}= M^{3}_{3,9},$ and $M^{5}= M^{4}_{3,11}.$

$$	A = \kbordermatrix{
	& 1 & 2 & 3 & 4 & 5 & 6 & 7 & 8 & 9 & 10 & 11 & 12 \\
	&1 & 0 & 0 & 0  & 0 & 0 & 0 & 1 & 1 & 1 & 1 & 1\\
	&0 & 1 & 0 & 0  & 0 & 0 & 1 & 0 & 1 & 2 & 2 & 1\\
	&0 & 0 & 1 & 0  & 0 & 0 & 1 & 1 & 0 & 1 & 2 & 2\\
	&0 & 0 & 0 & 1  & 0 & 0 & 1 & 2 & 1 & 0 & 1 & 2\\
	&0 & 0 & 0 & 0  & 1 & 0 & 1 & 2 & 2 & 1 & 0 & 1\\
	&0 & 0 & 0 & 0  & 0 & 1 & 1 & 1 & 2 & 2 & 1 & 0 
} $$

The matroid $M^5$ is Eulerian in which $E(M)$ forms a circuit. We observe that each matroid $M$ and $M^k, k \in \{1,2,3,4\},$  have a $p$-\textit{decomposition}. Therefore by Proposition \ref{e1} the matroids $M^k, k \in \{1,2,\ldots,5\}$ are Eulerian.
\end{example}


\noindent In the light of Example \ref{counter}, the Theorem \ref{the1} by Raghunathan et al. \cite{rsw} does not hold, in general, for $p$-matroids, $p> 2.$

\noindent For $p>2,$ it remains open to prove or disprove the existence of a sequence of splitting operations on Eulerian $p$-matroid $M$ which transforms the gound set $E(M)$ into a circuit. The following algorithm transforms $E(M)$ into a circuit, if a sequence of splitting operations, described earlier,exists; otherwise, $E(M)$ can not be transformed into a circuit.  \\

\noindent \textbf{Algorithm:} \\ Let $M$ be Eulerian $p$-matroid, 
\begin{description}
	\item $\mathbf{Step~1:}$ Find all the circuits of $M.$
	\item $\mathbf{Step~2:}$ List all the circuit decompositions of $M$ as $\{\mathcal{D}_1, \mathcal{D}_2, \ldots, \mathcal{D}_k\}.$
	\item $\mathbf{Step~3:}$If there is a pair $a,b \in E(M)$ such that $\mathcal{D}_i$ is a $p$-decomposition of $M,$ and $|\mathcal{D}_i|>1,$ for some $i \in [k],$ then the splitting matroid $M_{a,b}$ is Eulerian by Proposition \ref{e1}. Replace $M$ by $M_{a,b}$ and go to $\mathbf{Step~1.}$ \\~\\If no such pair $a,b \in E(M)$ exists, then go to \textbf{Step~4}.
	\item $\mathbf{Step~4:}$ Then $E(M)$ is the circuit of $M$ or $E(M)$ can not be transformed into a circuit.
\end{description}



\section*{Acknowledgment}

This paper is dedicated to the memory of T.T. Raghunathan who recently passed away on December 18, 2021.

\end{document}